\documentclass[a4paper,11pt]{amsart}
\usepackage{hyperref}
\usepackage[linesnumbered,ruled,vlined]{algorithm2e}

\usepackage[usenames,dvipsnames,svgnames,table]{xcolor}
\usepackage[all]{xy}
\usepackage{enumitem}
\usepackage[normalem]{ulem}
\newcommand{\tsk}[1]{\textcolor{YellowOrange}}

\usepackage{tikz}
\usepackage{tikz-cd}

\makeatletter
\def\@endtheorem{\endtrivlist}
\makeatother

\usepackage[active]{srcltx}
\usepackage{amsmath}
\usepackage{amssymb}
\usepackage{amscd}
\usepackage{amsthm}
\usepackage[utf8]{inputenc}
\usepackage{mathrsfs}  

\usepackage{nicefrac}

\newtheorem{teo}{Theorem}[section]
\newtheorem{problem}{Problem}[section]

\newtheorem{definition}[teo]{Definition}
\newtheorem{prop}[teo]{Proposition}
\newtheorem{cor}[teo]{Corollary}
\newtheorem{lemma}[teo]{Lemma}
\newtheorem{conjecture}[teo]{Conjecture}
 
\theoremstyle{definition}

\newtheoremstyle{dico}
 {\baselineskip}   
  {\topsep}   
  {}  
  {0pt}       
  {} 
  {.}         
  {5pt plus 1pt minus 1pt} 
  {}          
\theoremstyle{dico}
\newtheorem{say}[teo]{}
\numberwithin{equation}{section}

\newcommand{\ra}{\rightarrow}
\newcommand{\C}{\mathbb{C}}
\newcommand{\R}{\mathbb{R}}
\newcommand{\Zeta}{{\mathbb{Z}}}
\newcommand{\N}{{\mathbb{N}}}

\newcommand{\meno}{^{-1}}

\newcommand{\alfa}{\alpha}
\newcommand{\alf}{\alpha}
\newcommand{\vacuo}{\emptyset}

\newcommand{\enf}{\emph}

\newcommand{\Aut}{\operatorname{Aut}}

\newcommand{\End}{\operatorname{End}}

\newcommand{\ord}{{\operatorname{ord}}}

\newcommand{\eps}{\varepsilon}
\renewcommand{\phi}{\varphi}
\newcommand{\lds}{\dotsc}
\newcommand{\cds}{\cdots}
\newcommand{\cd}{\cdot}

\newcommand{\sx}{\langle}
\newcommand{\xs}{\rangle}

\newcommand{\lra}{\longrightarrow}

\newcommand{\ga}{\gamma}
\newcommand{\Ga}{\Gamma}

\newcommand{\fix}{\operatorname{Fix}}

\newcommand{\Gl}{\operatorname{GL}}

\newcommand{\PP}{\mathbb{P}}   
   
\renewcommand{\phi}             {\varphi}

\newcommand{\M}{\mathsf{M}}

\newcommand{\T}{\mathsf{T}}
\newcommand{\Mg}{\mathsf{M}_g}
\newcommand{\Datum}{\Delta}
\newcommand{\Da}{\Delta}

\newcommand{\td}{\mathsf{T}_\Datum}
\newcommand{\md}{\mathsf{M}_\Datum}
\newcommand{\tmd}{\tilde{\mathsf{M}}_\Datum}
\newcommand{\ttmd}{\mathsf{B}_\Datum}

\newcommand{\zd}{\mathsf{Z}_\Datum}
\newcommand{\cc}{\mathscr{C}}
\newcommand{\A}{\mathsf{A}}
\newcommand{\Ag}{\mathsf{A}_g}
\newcommand{\mm}{{\mathbf{m}}}
\newcommand{\tor}{\overline{j(\mg)}} 

\newcommand{\ag}{\mathsf{A}_g}
\newcommand{\mg}{\mathsf{M}_g}

\newcommand{\datum}{{( G, \theta)}}

\newcommand{\mihi}[1]{}
\newcommand{\prym}{\mathscr{P}}
\newcommand{\Nm}{\operatorname{Nm}}

\newcommand{\Mod}{\operatorname{Mod}}
\newcommand{\gtheta}{{G_\theta}}

\newcommand{\magma}{{\texttt{MAGMA} \cite{MA}}}

\newcommand{\normale}{\vartriangleleft}

\newcommand{\barc}{\bar{\chi}}
\newcommand{\baralf}{\bar{\alf}}

\newcommand{\data}{\mathscr{D}}
\newcommand{\datar}{\mathscr{D}^r}

\newcommand{\abs}[1]{\left\vert#1\right\vert}
\newcommand{\irr}{\operatorname{Irr}}

\newcommand{\bart}{\bar{\theta}}
\newcommand{\dd}{{\bar{\Delta}}}
\newcommand{\ttmdd}{\mathsf{B}_\dd}
\newcommand{\muu}{m}
\newcommand{\Ab}{\operatorname {Ab}}

\begin{document}

\author{Diego Conti, Alessandro Ghigi, Roberto Pignatelli}
\title
{Some evidence for the Coleman-Oort conjecture}

\address{Universit\`{a} di Milano - Bicocca}
\email{diego.conti@unimib.it}

\address{Universit\`{a} di Pavia}
\email{alessandro.ghigi@unipv.it}

\address{Universit\`{a} di Trento}
\email{Roberto.Pignatelli@unitn.it}

\subjclass[2020]{Primary: 14G35, 
  14J10, 
  14Q05, 
Secondary:  20F99, 
 }

 \thanks{
The authors were partially supported by  INdAM (GNSAGA).
   The second author was partially supported also by MIUR PRIN 2015
  ``Moduli spaces and Lie Theory'' , by MIUR FFABR, by FAR 2016
  (Pavia) ``Variet\`a algebriche, calcolo algebrico, grafi orientati e
  topologici'', by MIUR, Programma Dipartimenti di Eccellenza
  (2018-2022) - Dipartimento di Matematica ``F. Casorati'',
  Universit\`a degli Studi di Pavia. The third author was partially supported also by MIUR PRIN 2015
  ``Geometry of Algebraic Varieties''  and by MIUR PRIN 2017
  ``Moduli Theory and Birational Classification''.}

\maketitle

\begin{abstract}
  The Coleman-Oort conjecture says that for large $g$ there are no
  positive-dimensional Shimura subvarieties of $\mathsf{A}_g$
  generically contained in the Jacobian locus. Counterexamples are
  known for $g\leq 7$.  They can all be constructed using families of
  Galois coverings of curves satisfying a numerical condition.  These
  families are already classified in cases where: a)~the Galois group
  is cyclic, b)~it is abelian and the family is 1-dimensional, and
  c)~$g\leq 9$. By means of carefully designed computations and
  theoretical arguments excluding a large number of cases we are able
  to prove that for $g\leq 100$ there are no other families than those
  already known.
\end{abstract}

\maketitle

\section{Introduction}

\begin{say}

  Denote by $\ag$ the moduli space of principally polarized complex
  abelian varieties of dimension $g$, by $\mg$ the moduli space of
  smooth complex algebraic curves of genus $g$ and by
  $j \colon \mg \ra \ag$ the period mapping (or Torelli mapping),
  which associated to $[C] \in \Mg$ the moduli point of the Jacobian
  variety $JC$ provided with the theta polarization.  The
  \emph{Jacobian locus} is the image $j (\mg)$.
  By $\tor$ we denote the closure of $j(\mg)$ in $\ag$.
  
  On $\ag$ there is a tautological $\mathbb{Q}$-variation of the Hodge
  structure (in the orbifold sense): if $A$ is a principally polarized
  abelian variety, the fibre over its moduli point $[A] \in \Ag$ is
  $H^1(A, \mathbb{Q})$ with its Hodge structure of weight 1.  In
  general, given a variation of the Hodge structure $H \ra B$, it is
  interesting to consider the points $b \in B$ where the Hodge
  structure is ``more symmetric'' than over the general point. Making
  precise the meaning of ``more symmetric'' requires some effort. In
  the simplest case this means that the Hodge structure has more
  automorphism than usual. For example for the variation over $\A_1$,
  the general point has no automorphisms beyond $\{\pm 1\}$, while the
  points with more automorphisms represent the well-known elliptic
  curves with automorphisms $\Zeta/4\Zeta$ or $\Zeta/6\Zeta$.
  The general case is more complicated since the symmetry is not at
  the level of automorphisms but is detected by Hodge classes in
  general tensor spaces. The loci obtained in this way are called the
  \emph{Hodge loci} of the variation of the Hodge structure. In the
  case of $\Ag$ they are also called \emph{special subvarieties} or
  \emph{Shimura subvarieties}. (See \cite[\S 3.3]{moonen-oort} and
  \cite{deba}.)  A subvariety $Z \subset \A_g$ is said to be
  \emph{generically contained} in $j(\Mg)$ if $Z \subset \tor$ and
  $Z \cap j(\Mg)\neq \vacuo$.  Arithmetical considerations led first
  Coleman and later Oort \cite{oort-can} to the following
\end{say}
\begin{conjecture}[Coleman-Oort]
  For large $g$ there are no special subvarieties of positive
  dimension generically contained in $j(\Mg)$.
\end{conjecture}
(See \cite[\S 4]{moonen-oort} for more details.)  This expectation is
also motivated by another stronger expectation originating from the point of view of
differential geometry: special subvarieties are totally geodesic with
respect to the locally symmetric (orbifold) metric on $\ag$ (the one
coming from the Siegel space). If one believes that $j(\mg)$ bears no
strong relation to the ambient geometry of $\ag$, in particular that
it is very curved inside $\ag$, then it is natural to expect that
$j(\Mg)$ contains generically no totally geodesic subvarieties, and in
particular no Shimura subvarieties (see 
\cite{cfg}, \cite{gpt}, \cite{fpi} for results in this direction).

What makes the problem more interesting is that for low genus examples
of such Shimura varieties generically contained in $j(\mg)$ do exist!
All the examples known so far are in genus $g \leq 7$ and arise from
one of the following two constructions.

\begin{say}
  \label{sayuno}
  \textbf{First construction}.  Let $G$ be a finite group acting on a
  curve $C$.  Consider the family of curves $\mathscr{C} \ra B$ with a
  $G$-action of the same topological type (see below for the precise
  definition).  For every $m$, $H^0(C_b,mK_{C_b})$ is a representation
  of $G$ and its equivalence class is independent of $b\in B$.  Denote
  by $B' \subset \M_g$ the moduli image of $B$ and by $Z$ the closure
  of $j(B')$ in $\A_g$.  In \cite{fgp,fpp} it is proven that if
  \begin{gather}
    \label{star}
    \tag{$\ast$} \dim (S^2 (H^0(K_{C_b})))^G = \dim H^0(2K_{C_b})^G,
  \end{gather}
  then $Z$ is a Shimura variety generically contained in $j(\mg)$.  We
  also say that the family of $G$-covers $\mathscr{C} \ra B$
  \emph{yields} a Shimura variety to mean that $Z$ is Shimura. We
  refer to such a Shimura variety as a \emph{counter-example} to
  Coleman-Oort conjecture.  Several counter-examples are known, see
  Theorem \ref{data} below.
\end{say}

\begin{say}
  \label{saydue}
  \textbf{Second construction}. Consider a Shimura variety $Z$
  generically contained in $j(\Mg)$ obtained as in \ref{sayuno} from a
  family of $G$-curves $\mathscr{C} \ra B$. Denote by $g'$ the genus
  of $C_b/G$. Let $\Nm: JC_b \ra J(C_b/G)$ be the norm map of the
  covering $f_b : C_b \ra C_b/G$, defined by
  $\Nm( \sum_i p_i):= \sum_i f_b(p_i)$, Then $(\ker \Nm)^0\subset JC_b$
  is an abelian subvariety, the generalized Prym variety of the
  covering $f_b$.  The theta polarization of $JC_b$ restricts to a
  polarization of some type $\delta$ on the Prym variety. We get maps
  \begin{gather*}
    \phi: B \lra \Mg, \quad   \phi(b):= [C_b/G] , \\
    \prym: B \lra \A_{g-g'}^\delta , \quad \prym (b):= [ (\ker \Nm)^0
    ] .
  \end{gather*}
  $\prym$ is the \enf{generalized Prym map}.  If $g'=0$ the map $\phi$
  is of course constant, $\A_{g-g'}^\delta = \Ag$ and $\prym $ is just
  the Torelli map, so we get nothing new.  If instead $g' > 0$, the
  irreducible components of the fibres of $\prym$ and $\phi$ are
  totally geodesic subvarieties and countably many of them are in fact
  Shimura, see \cite{gm2} and \cite[Thm. 3.9, Thm. 3.11] {fgs}.  Thus
  for $g'>0$ this construction gives uncountably many totally geodesic
  non-Shimura varieties generically contained in $j(\Mg)$ and
  countably many Shimura varieties generically contained in $j(\Mg)$.
\end{say}

Let us summarize what is known about the counter-examples obtained via
these constructions.
\begin{teo}\label{data}
  \begin{enumerate}
  \item [a)] There are 38 families of Galois coverings of the
    projective line satisfying \eqref{star} with $2 \leq g \leq 7$.
    For $g \leq 9$ there are no other counter-examples.  See
    \cite{rohde,moonen-special,moonen-oort,fgp}.
  \item [b)] There are 6 families of Galois coverings of elliptic
    curves satisfying \eqref{star} with $2 \leq g \leq 4$.  For
    $g \leq 9$ there are no other counter-examples.  See \cite{fpp}.
  \item [c)] If a family satisfies \eqref{star} and $g'>0$, then
    necessarily $g'=1$ and the family is one of those in (b). See
    \cite{fgs}.
  \end{enumerate}
\end{teo}
\begin{say}
Note that   we focus on $g\geq 2$, since for $g=1$ there are infinitely many
1-dimensional families satisfying \eqref{star}. 

In fact, for every elliptic curve $C$ the involution $p \mapsto -p$ acts trivially on both $S^2H^0(K_C)$ and $H^0(2K_C)$.
Let $G$ be the group of the biholomorphisms of $C$ generated by it and by a finite group of translations. 
Then $S^2H^0(K_C)^G=S^2H^0(K_C)\cong \C \cong H^0(2K_C)=H^0(2K_C)^G$, so giving examples of \eqref{star} with $G$ of order arbitrarily high.
 Two of these families are listed in Table
2 in \cite{fgp}. 

However all these families are irrelevant for the Coleman-Oort conjecture, since in all cases $B' =\M_1$.
  Note also  that some of the families of Theorem \ref{data} yield the same
  Shimura variety, i.e. have the same image in moduli, see \cite {fgp,fpp}.
\end{say}

\begin{say}\label{solo1}
  It follows from Theorem~\ref{data} (c) that all the cases where \eqref{star} holds and $g'>0$
  are already known and also that no new examples can be found using
  the second construction \ref{saydue}.  Therefore, in order to
  construct new examples using the two methods above (or to exclude
  the existence of such examples) we can restrict to the first
  construction with $g'=0$, i.e. $C_b/G = \PP^1$.
\end{say}

The purpose of this paper is to provide further evidence for the Coleman-Oort conjecture, employing  a computational approach complemented by theoretical arguments. Our result is the following improvement of Theorem
\ref{data}.

\begin{teo} \label{main} The positive-dimensional families of Galois
  covers satisfying \eqref{star} with $2\leq g\leq 100$ are only those
  of Theorem \ref{data}.
\end{teo}

\begin{say}
  
  The fact that we found \textbf{no} new families at all is strong
  evidence that there are no more families satisfying (\ref{star}).
  Since all known counter-examples to the Coleman-Oort conjecture can be
  constructed using these families, this also suggests that either further counter-examples do not exist or they are of a completely different nature.
\end{say}

\begin{say} An important point to stress is the following.  Condition
  \eqref{star} is sufficient for a family to yield a Shimura variety.
  In general it is unknown if it is also necessary.
  In this paper we only check whether condition \eqref{star} holds.
  So we cannot exclude that these families give rise to
  counter-examples to Coleman-Oort conjecture.
\end{say}

\begin{say}
  Families of $G$-covers are identified by data of combinatorial and
  group-theoretical nature. We explain this in \S \ref{sec-data}.  So
  the basic strategy is obviously to list all these data and check
  condition \eqref{bona} for each datum in the list.  Since the list
  of these data is extremely long, one needs to avoid unnecessary
  computations.  The first observation is that many data give rise to
  the same family.  More precisely call two data $\Datum$ and
  $\Datum'$ \emph{Hurwitz equivalent} if they have the same group $G$
  and if the families corresponding to them are isomorphic as families
  of algebraic curves with $G$-action.  It turns out that Hurwitz
  equivalence classes can be huge.  To check condition \eqref{star}
  for all the families of some genus, one would start by choosing a
  representative out of any Hurwitz equivalence class, and proceed by
  checking \eqref{star} for all the representatives. However, the
  identification of a single representative inside each class is a
  daunting task, since the classes are huge and Hurwitz equivalence is
  rather complicated.  (An algorithm dealing with Hurwitz equivalence
  appears in \cite{BCGP}. It was used in \cite{fgp} and \cite{fpp}. An
  improvement of this algorithm is given in \cite{BP12}. We hope to
  address the problem of algorithmic computation of Hurwitz
  equivalence in future work.)

  Luckily there is another equivalence relation on data, much coarser
  than the Hurwitz equivalence, which is appropriate to our problem:
  if $\Datum=(G,g_1, \lds, g_r)$, then the number $N=N(\Datum)$ only
  depends on the conjugacy classes $C_1=[g_1], \lds, C_r=[g_r]$.  Also
  the order of these is completely irrelevant. The unordered sequence
  $(C_1, \lds, C_r)$ is called a \emph{refined passport}.  (See
  Definition \ref{defrefined}.)  So our problem depends only on
  refined passports, more precisely on their $\Aut(G)$-orbits, which
  are considerably less in number than Hurwitz equivalence classes,
  leading to much shorter execution times. Notice that in some cases
  refined passports (even if taken up to the action of $\Aut(G)$) are
  still too many to be stored simultaneously into memory, but this is
  not a problem, since we only need to perform an iteration to check
  \eqref{star} on each individually.

  Even after this great simplification the computation remains quite
  formidable, at least for the computers at our disposal. We use a
  number of tricks to reduce the data that must be considered. Several
  exclusions (e.g. cyclic groups) follow from previous results (see
  Theorem \ref{mona}). We complement them with Corollary
  \ref{strongerthanall}, which effectively eliminates more than $90\%$ of
  the data, including some of the hardest cases, thus allowing us to
  complete the computation.
\end{say}

\begin{say}
  For the implementation of the algorithm we used \magma, which is
  quite suited to the task at hand since it allows working with
  groups, group actions and representations, in particular computing
  characters, orbits and stabilizers; furthermore, it contains a
  database of groups of small order.  Our code is available at
  \cite{diecicentomillecolemanoort}.

  The problem lends itself easily to parallelization, since each group
  and signature is treated independently; however, \texttt{MAGMA} does
  not support parallelization natively. The first part of the computation (Algorithm~\ref{algo:signatures}) was parallelized using the standard tool  \cite{Tange2011a}. On the other hand, the rest of the computation can become quite memory-intensive; this leads to technical difficulties,
  mainly concerning situations in which one of the processes is
  terminated for lack of memory, which were addressed by writing the ad
  hoc external program \cite{hlidskjalf} to run the \texttt{MAGMA} script.
  
  Using a computer with 56 Intel Xeon 2.60GHz CPU and 128 GB of RAM we
  were able to finish the computations in less than three days.
\end{say}

\begin{say}
  The plan of the paper is as follows.  In \S \ref{sec-data} we recall
  the description of the families of $G$-curves and some basic facts
  concerning the multiplication map on sections of the canonical
  bundle, which is related with condition \eqref{star}.  At the end we
  prove Lemma \ref{quoziente}, which deals with the behaviour of
  condition \eqref{star} when passing from a given family to a
  quotient by a normal subgroup.  In \S \ref{sec:tricks} we gather
  several facts of quite different nature, some well-known, some new,
  which we have found useful to rule out several cases.  This has been
  essential in order to complete the computation.
  Finally \S \ref{sec:algoritmo} contains a thorough
  explanation of the algorithm.
\end{say}

\medskip

{\bfseries \noindent{Acknowledgements.} } The authors would like to
thank Paola Frediani for help with Lemma \ref{quoziente} and Matteo
Penegini and Fabio Perroni for interesting discussions related to the
subject of this work. The second author would like to thank Matteo
Garofano and Gabriele Merli for technical help with the installation
and the maintenance of the server used for the computations.

\section{Families of $G$-curves}
\label{sec-data}

\begin{say}
  The purpose of this section is to describe some group-theoretic and
  combinatorial data from which one can construct algebraic families
  of curves with prescribed symmetry. We will denote by $\Datum$ the
  datum and by $\mathscr{C}_\Datum \ra B_\Datum$ the corresponding
  family of curves.  The image of $B_\Datum$ in $\mg$ will be denoted
  by $\md$.  We are interested in the closure of $\md$ in $\ag$. As
  explained in \ref{sayuno}, when \eqref{star} holds this closure is a
  Shimura variety generically contained in the Jacobian locus. This is
  explained in more detail at the end of this section, together with
  some related remarks on the multiplication map.
  
  In the following, unless otherwise stated, we assume that the genus
  is at least 2.  For $r \geq 3$, set
  \begin{gather*}
    \Ga_{r}:= \langle \gamma _1, \dots, \gamma _r \ \vert \ \prod
    _{i=1} ^r \gamma _i = 1 \rangle.
  \end{gather*}
\end{say}

\begin{definition}
  \label{def:data}
  If $G$ is a finite group an epimorphism $\theta : \Ga_r \ra G$ is
  called \emph{admissible} if $\theta(\ga_i)\neq 1$ for
  $i=1, \lds, r$.  An $r$-datum is a pair $\Da=(G, \theta)$ where $G$
  is a finite group and $\theta : \Ga_r \ra G$ is an admissible
  epimorphism.  The \emph{signature} of $\Da$ is the vector
  $\mm:=(m_1, \lds, m_r)$ where $m_i:=\ord (\theta (\ga_i))$.  The
  \emph{genus} of $\Da$, denoted by $g(\Delta)$, is defined by the
  Riemann-Hurwitz formula:
  \begin{gather}
    \label{RH}
    2(g(\Delta)-1)=|G|\left(-2+\sum_{i=1}^{r}\left(1-\frac{1}{m_{i}}\right)\right)
  \end{gather}
  We let $\datar$ or simply $\data$ denote the set of all $r$-data.
\end{definition}

\begin{say}
  \label{geometric-basis}
  Orient $S^2$ by the outer normal.  Consider smooth regular arcs
  $\tilde{\alpha}_i$ in $S^2$ joining $p_0$ to $p_1$ such that for
  $i\ne j$ $\tilde{\alfa}_i$ and $\tilde{\alfa}_j$ intersect only at
  $p_0$. Assume also that the tangent vectors at $p_0$ are all
  distinct and follow each other in counterclockwise order.  Next
  consider loops $\alfa_i$ based at $p_0$ constructed as follows:
  $\alfa_i$ starts at $p_0$, travels along $\tilde{\alfa}_i$ until
  near $p_i$, there travels counterclockwise along a small circle
  around $p_i$, finally goes back to $p_i$ again along
  $\tilde{\alfa}_i$. The circles have to be pairwise disjoint. We call
  the resulting set of generators $\{[\alfa_1], \lds, [\alfa_r]\}$ a
  \emph{geometric basis} of $\pi_1(S^2 -P,p_0)$.  Once a geometric
  basis is fixed, there is a well-defined isomorphism
  \begin{gather*}
    \chi : \Ga_r \ra \pi_1(S^2 - P, p_0)
  \end{gather*}
  such that $\chi(\ga_i) = [\alfa_i]$.
\end{say}

\begin{say}
  The following geometric setting gives rise to data (and it is the
  main motivation for them). Let $X$ be a compact (connected) Riemann
  surface. Assume that a finite group $G$ acts effectively and
  holomorphically on $X$ in such a way that $X/G =\PP^1$.  Let
  $P:=\{p_1, \lds, p_r \} $ be the critical values of
  $\pi : X \ra \PP^1\cong S^2$.  Fix $p _0 \in S^2 -P$ and a geometric
  basis $\{[\alf_1], \ldots , [\alf_r]\}$ with corresponding
  isomorphism $\chi: \Ga_r \cong \pi_1 (S^2 -P, p_0)$.  Finally fix a
  point $\tilde{p}_0 \in \pi\meno(p_0)$. As is well-known there is a
  morphism $\tilde{\theta} : \pi_1(S^2-P,p_0) \ra G$ such that for
  $ [\alfa ] \in \pi_1(S^2-P, p_0)$ the lifting of $\alfa$ starting at
  $p_0$ ends at $ g\cd p_0$ where $g = \bar{\theta} ([\alfa])$.  Since
  $X$ is connected $\tilde{\theta}$ is surjective. Therefore
  $\Da:=(G,\theta:=\tilde{\theta}\circ \chi)$ is an $r$-datum,
  $g(\Da) = g(X)$ by the Riemann-Hurwitz formula and $m_i$ is the
  cardinality of the stabilizer of points in $\pi\meno(p_i)$.  We are
  going to show that each datum arises from a covering
  $X \ra \PP^1=X/G$.
\end{say}

\begin{say}
  Assume from now on that $r\geq 3$ and denote by $\T_{0,r}$ the
  Teichm\"uller space in genus $0$ and with $r$ marked points.  The
  definition of $\T_{0,r}$ is as follows.  Fix $r+1$ distinct points
  $p_0, \lds, p_r$ on $S^2$.  For simplicity set
  $P = (p_1, \lds, p_r)$.  Consider triples of the form
  $(\PP^1, x, [f])$ where $x = (x_1, \lds, x_r) $ is an $r$-tuple of
  distinct points in $\PP^1$ and $[f]$ is an isotopy class of
  orientation preserving homeomorphisms
  $f : (\PP^1, x) \ra (S^2 , P)$.  Two such triples $(\PP^1, x , [f])$
  and $(\PP^1, x', [f'])$ are equivalent if there is a biholomorphism
  $\phi: \PP^1 \ra \PP^1$ such that $\phi(x_i) = x'_i$ for any $i$ and
  $ [f] = [f'\circ \phi ]$. The Teichm\"uller space $\T_{0,r}$ is the
  set of all equivalence classes, see e.g. \cite[Chap. 15]{acg2} for
  more details.
\end{say}

\begin{say} \label{saydatum} Fix a geometric basis
  $\mathscr{B}=\{[\alfa_i]\}_{i=}^r$ of $\pi_1(S^2 - P, p_0)$ with
  corresponding isomorphism $\chi: \Ga_r \cong \pi_1(S^2 - P, p_0)$.
  Given an $r$-datum $\Datum=(G, \theta)$, the epimorphism
  $\theta\circ \chi\meno$ gives rise to a topological covering
  $\pi:\Sigma_0 \ra S^2 -P$.  By the topological part of Riemann's
  Existence Theorem this can be completed to a branched cover
  $\pi: \Sigma \ra S^2$.  Given a point
  $t=[\PP^1, x, [f]] \in \T_{0,r}$, the homeomorphism $f$ restricts to
  a homeomorphism of $\PP^1 - x$ onto $S^2 - P$. We get an induced
  isomorphism
  $f_* : \pi_1(\PP^1 - x, f\meno(p_0)) \cong \pi_1(S^2 - P,
  p_0)$. Thus
  $\theta \circ \chi\meno \circ f_* : \pi_1(\PP^1-x,f\meno(p_0)) \ra
  G$ is an epimorphism and this gives rise to a topological covering
  $\pi_t^0: C^0_t \ra \PP^1-x$.  Here $C^0_t$ is an open
  differentiable surface.  Since $\pi_0$ is a local diffeomorphism,
  there is a unique complex structure on $C_t^0$ making $\pi^0_t$
  holomorphic.  By the holomorphic part of Riemann's Existence Theorem
  $C_t^0$ and $\pi_t^0$ may be uniquely completed to a proper
  holomorphic map $\pi_t \colon C_t \rightarrow \PP^1$ and the
  $G$-action extends to $C_t$.  Moreover there is an isotopy class of
  homeomorphisms $\tilde{f}_t : C_t \ra \Sigma$ that cover $f_t$.  As
  $t$ varies in $ \T_{0,r}$ this construction yields a holomorphic map
  to the Teichm\"uller space of $\Sigma$
  \begin{gather*}
    \Phi_\Datum: \T_{0,r} \lra \T_g \cong \T(\Sigma), \quad t \mapsto
    [C_t,[\tilde{f}_t]].
  \end{gather*}
  The group $G$ embeds in the mapping class group of $\Sigma$, which
  we denote by $\Mod_g$.  This embedding depends on $\theta$ and we
  denote by $\gtheta \subset \Mod_g$ its image.  The image of
  $\Phi_\Datum$ coincides with $\T_g^{G_\theta}$, the set of fixed
  points of $\gtheta$ on $\T_g$. As such it is a complex submanifold.
  We denote it by $\T_\Datum$.

  The image of $ \T_\Datum$ in the moduli space $\M_g$ is an
  irreducible algebraic subvariety of dimension $(r-3)$ that we denote
  by $\mathsf{M}_ \Datum$.  (See e.g.
  \cite{gavino,baffo-linceo,BCGP,clp} for more details.)  As explained
  in \cite[p. 79]{gavino} the map $\td \ra \md$ factors through an
  intermediate variety $\tmd$:
  \begin{gather*}
    \T_\Datum \lra \tmd \stackrel{\nu}{\lra} \md.
  \end{gather*}
  The variety $\tmd $ is the normalization of $\md$.
  There is a finite cover $\ttmd \ra \tmd$ and a universal family
  \begin{gather*}
    \pi_\Datum: \cc _\Datum \ra \ttmd.
  \end{gather*}
  We call it the \emph{family of $G$-curves associated to} $\Datum$.
  The proofs of these assertions can be found in \cite{gavino} (where
  $T_g(H_0)$ corresponds in our notation to $\td$,
  $\widetilde{\mathcal{M}}(H_0)$ to $\tmd $, $\mathcal{M}(H_0)$ to
  $\md $ and $\tilde{\mathcal{M}}^{pure}(H_0)$ to $\ttmd$).  Note that
  \begin{gather}\label{r3}
    \dim \md= \dim \ttmd = r-3.
  \end{gather}
\end{say}

\begin{say}
  \label{normale}
  In this construction the choice of the base point $p_0$ is
  irrelevant.  In fact (up to isomorphism) the ramified covering
  $\Sigma \ra S^2$ only depends on
  $N:=\ker \,\theta \circ \chi\meno \normale \pi_1(S^2-P,p_0)$. Two
  isomorphism $\pi(S^2-P,p_0) \ra \pi_1(S^2-P,p_0')$ differ by an
  inner automorphism, so the map from normal subgroups of
  $\pi(S^2-P,p_0)$ to those of $\pi(S^2-P,p_0')$ is well defined. This
  proves that $\td$ and hence also $\md$, $\tmd$ and the family
  $\pi_\Datum: \cc_\Datum \ra \ttmd$ do not depend on the choice of
  the base point $p_0$.
\end{say}

\begin{say} \label{sayeta} On the other hand the construction of
  $\td, \md, \tmd, \pi_\Datum$ does depend on the choice of the
  geometric basis.  Let
  $\overline{\mathscr{B}}=\{[\baralf_i]\}_{i=1}^r$ be another geometric
  basis.  and let $\barc: \Ga_r \ra \pi_1(S^2 -P, p_0)$ be the
  corresponding isomorphism.  Then
  $\mu:=\barc \circ \chi\meno \in \Aut \pi(S^2 -P,p_0)$ has two
  special properties: 1) for every $i = 1, \lds, r$,
  $\mu([\alfa_i])=[\baralf_i]$ is conjugate to $[\alfa_j]$ for some
  $j$; 2) the induced homomorphism on the cohomology group
  $H_2(\pi_1(S^2-P,p_0), \Zeta)$ is the identity.  By a variant of the
  Dehn-Nielsen Theorem (see e.g. \cite[\S 8.2.7 p. 233]{fmarga} or
  \cite[Thm. 5.7.1 p. 197]{zieschang}) there is an
  orientation-preserving diffeomorphism
  $\phi : (S^2 -P , p_0) \ra (S^2 -P , p_0) $ such that
  $\mu = \phi_*$.  Let $\Sigma $ and $\bar{\Sigma}$ be the coverings
  of $S^2$ obtained from $\chi$ and $\barc$.  If
  $N =\ker\, \theta \circ \chi\meno $ and
  $\bar{N} =\ker\, \theta \circ (\barc)\meno $, then
  $\phi_* (N) = \bar{N}$. By the Lifting Theorem there is an
  orientation-preserving diffeomorphism
  $\tilde{\phi} : \Sigma \ra \bar{\Sigma}$ that covers $\phi$. This
  gives rise to a biholomorphism $\T (\Sigma ) \ra \T (\Sigma')$ which
  maps $\td$ constructed using $\chi$ to $\td$ constructed using
  $\barc$.  The identification $\T_g = \T(\Sigma)$ is defined up to
  the action of $\Mod _g$ and the discussion above shows that also
  $T_\Datum$ is well defined up to this action.  In particular
  $\md , \tmd, \ttmd$ and $\pi_\Datum$ are completely independent of
  the choice of the geometric basis.
\end{say}

\begin{say}
  \label{sayrho}
  There is a representation
  \begin{gather}
    \label{defrho}
    \rho: G \lra \Gl H^0(C_t, K_{C_t}), \quad \rho (g) :=(g\meno)^* .
  \end{gather}
  The equivalence class of this representation is independent of
  $t\in \ttmd$.
  
  For later use we recall the following observation, already used in
  the proof of \cite[Thm. 2.3]{fgs}.
  
\end{say}
\begin{prop}\label{suriettaquasisempre}
  Let $G$ be a finite group of automorphisms of a curve $C$, and
  consider the subspace of invariants $H^0(C, 2 K_{C})^G$.  Then the
  multiplication map
  \[m_C^G \colon S^2H^0(C,K_C)^G \rightarrow H^0(C,2K_C)^G
  \]
  is surjective unless $C$ is hyperelliptic (so of genus at least $2$)
  and there is a small deformation $C_t$ of the complex structure of
  $C$ such that all elements of $G$ remain holomorphic and the general
  curve $C_t$ is not hyperelliptic.
  
  In particular, for a fixed  $r$-datum $\Datum=(G,\theta)$, the map $m_C^G$
  is surjective for the general $C \in \ttmd$.
\end{prop}
\begin{proof}
  Let $g$ be the genus of $C$.  The statement is obvious for
  $g \leq 1$ since the $G-$equivariant map
  $S^2(H^0(C,K_C)) \rightarrow H^0(C,2K_C)$ is an isomorphism (among
  spaces of dimension $g$).  If $C$ is not hyperelliptic, then the
  statement follows similarly since the map
  $S^2(H^0(C,K_C)) \rightarrow H^0(C,2K_C)$ is surjective by
  M. Noether's Theorem.
  
  We can then assume that $C$ is hyperelliptic. Let $\sigma$ be the
  hyperelliptic involution. It is well-known that $\sigma$ acts as the
  multiplication by $-1$ on $H^0(C,K_C)$, so trivially on
  $S^2(H^0(C,K_C))$, and that the multiplication map
  $ S^2(H^0(C,K_C)) \rightarrow H^0(C,2K_C)^{\left\langle \sigma
    \right\rangle}$ is surjective.
  
  We distinguish two cases.
   \begin{enumerate}
  \item If $\sigma \in G$ then the surjectivity of $m_C^G$ follows by
    the surjectivity of the map
    $ S^2(H^0(C,K_C)) \rightarrow H^0(C,2K_C)^{\left\langle \sigma
      \right\rangle}$.
  \item If $\sigma \not\in G$ we denote by $\tilde{G}$ the group of
    automorphisms of $G$ generated by $G$ and $\sigma$.  Then
    $m_C^{\tilde{G}}$ is surjective. Moreover
    $ S^2(H^0(C,K_C))^{\tilde{G}} = S^2(H^0(C,K_C))^G$ so we need
    $ H^0(C,2K_C)^G \cong H^0(C,2K_C)^{\tilde{G}}$, that is equivalent
    to
    $ H^0(C,2K_C)^G\subset H^0(C,2K_C)^{\left\langle \sigma
      \right\rangle}$. Dualizing, this is equivalent to
    $ H^1(C,T_C)^G\subset H^1(C,T_C)^{\left\langle \sigma
      \right\rangle}$, which amounts to asking that every small
    deformation of the pair $(C,G)$ remain hyperelliptic.
  \end{enumerate}
\end{proof}
  
  \begin{say}
    We notice that the exceptional case in Proposition
    \ref{suriettaquasisempre} occurs.  Consider for example family
    (27) in \cite[Table 2]{fgp}.  A direct computation shows that this
    $3$-dimensional family of curves of genus $3$ with an action of
    $(\Zeta/2\Zeta)^2$ intersects the hyperelliptic locus in the
    2-dimensional family of curves with an action of
    $(\Zeta/2\Zeta)^3$ considered in \cite[Table 2 - Five critical
    values - (b)]{quasilarge}. If $C$ belongs to this latter family,
    then $3=h^0(C,2K_C)^G \neq h^0(C,2K_C)^{\tilde{G}}=2$ and
    therefore $m_C^G$ has corank $1$.
  \end{say}

\begin{say}
  Consider now a datum $\Datum$ and the family
  $\pi_\Datum: \cc_\Datum \ra \ttmd$. As $t$ varies in $\ttmd$, the
  domain and codomain of $m_{C_t}^G$ do not change in dimension. Set
  \begin{gather}
    \label{defN}
    N (\Datum): = \dim \left ( S^2H^0(C_t,K_{C_t})\right )^G
  \end{gather}
\end{say}
\begin{teo}
  \label{criterio0}
  If $g = g(\Datum) \geq 2$ and
  \begin{gather}
    \label{bona}
    \tag{$\ast$} N (\Datum) = r-3,
  \end{gather}
  then $\overline{j (\M_\Datum)} $ (closure in $\A_g$) is a special
  subvariety of PEL type of $\ag$ that is generically contained in the
  Jacobian locus.
\end{teo}
(See \cite[Thm. 3.9]{fgp} and \cite[Thm. 3.7]{fpp}.)

\begin{say}\label{uff}
  The idea of Theorem \ref{criterio0} is that from $\Datum$ one can
  construct both  $ \md$ and a Shimura subvariety
  $\zd \subset \ag$ with $N(\Datum) = \dim \zd$.  By construction
  $j ( \md) \subset \zd$ and both $\md$ and $\zd$ are irreducible
  algebraic subvarieties. By \eqref{r3} $\dim \md = r-3$.  Since $j$ is
  an injective morphism  of algebraic varieties, when $g\geq 2$ we always
  have $N \geq r- 3$.  If \eqref{bona} holds, then $j(\md) $ is dense
  in $\zd$.
\end{say}

\begin{say}\label{sayinjectiveisenough}
  Note also that (when $g\geq 2$) for any $t\in \ttmd$ we have
  $\dim H^0(2K_{C_t})^G = \dim H^1(T_{C_t})^G = \dim \ttmd = r-3$.
  Hence condition \eqref{bona} in Theorem \ref{criterio0} coincides
  with condition \eqref{star} of the Introduction. It amounts to
  asking that domain and co\-domain of $m_{C_t}^G$ have the same
  dimension.  By Proposition \ref{suriettaquasisempre} this is then
  equivalent to asking that, for general $t$, $m^G_{C_t}$ is
  injective.
\end{say}

      \begin{say}\label{sayquotientcovers}
        We now wish to prove a lemma that is helpful to rule out \emph{a priori} some
        groups. 

        Let $\Delta=(G,\theta)$ be a datum and let $H$ be a normal
        subgroup of $G$.  Set $K:=G/H$ and let $\pi: G \ra K$ be the
        canonical projection.  The composition
        $\pi\circ \theta:\Ga_r \ra G \ra K$ is an epimorphism, but it
        is not necessarily admissible, since some of the
        $\ga_i \in \Ga_r$ might map to $1$. We can throw them away
        obtaining an admissible epimorphism $\bart: \Ga_s \ra K$ for
        some $s \leq r$. In terms of spherical generators this means
        the following: if $\theta (\ga_i) = g_i$ and $k_i = \pi(g_i)$,
        then $\bart = (k_1, \lds, k_r)$ where we omit all the $k_i$
        that equal 1.  So we get a new datum $\dd = (K, \bart)$. This
        corresponds to the following geometric situation. $\Delta$
        gives rise to the family $\pi_\Datum: \cc_\Datum \ra \ttmd$.
        We can quotient each fibre $C_t$ by $H$ getting a curve
        $F_t:= C_t /H$ on which $K$ acts:
        \begin{equation*}
          \begin{tikzcd}
            C_t \arrow {dr} {\pi} \arrow{rr}{p}  &   &  F_t = C_t/H \arrow {dl} \\
            & \PP^1 =C_t/G=F_t/K . &
          \end{tikzcd}
        \end{equation*}
        The curves $F_t$ form a family $\mathscr{F} \ra \ttmd$.  If
        $g(F_t) \geq 2$, out of the datum $\dd$ we can form the family
        $\cc_\dd \ra \ttmdd$ as explained in \ref{saydatum}.  Then
        $\mathscr{F}$ is a pull-back of this family, i.e.
        $f^* \cc_\dd = \mathscr{F}$ for some holomorphic map
        $f: \ttmd \ra \ttmdd$.
      \end{say}

      \begin{lemma} \label{quoziente} In the above situation, assume
        that $g(F) \geq 2$. If \eqref{bona} holds for $\Delta$, then it
        holds also for $\dd$.
      \end{lemma}
      \begin{proof}
        Write for simplicity $C=C_t$ and $F=F_t$.
        We have two pull-back maps:
        \begin{gather*}
          p^* : H^0(K_F) \hookrightarrow H^0(K_C), \quad p^*:
          H^0(2K_F) \hookrightarrow H^0(2K_C).
        \end{gather*}
        From the first one we obtain also an injection
        \begin{gather*}
          f:=S^2p^*: S^2 H^0(K_F) \hookrightarrow S^2 H^0(K_C).
        \end{gather*}
        Since $p^* H^0(K_F)= H^0(K_C)^H$ then
        \begin{gather*}
          f ((S^2 H^0(K_F))^K) \subset (S^2H^0(K_C))^G.
        \end{gather*}
        Thus, we get a commutative diagram
        \begin{equation*}
          \begin{tikzcd}
            \left( S^2 H^0(K_F)\right)^K \arrow[hook]{d}[swap]{f}
            \arrow {rr}{\muu_{F}^K} & & H^0(2K_F)^K
            \arrow[hook]{d}{p^*}
            \\
            \left( S^2 H^0(K_C)\right)^G \arrow {rr}{\muu_{C}^G} & &
            H^0(2K_C)^G.
          \end{tikzcd}
        \end{equation*}
        from which
        \[
          m^G_C \text{ injective} \Rightarrow m^K_F \text{ injective}.
        \]

        As explained in \ref{sayinjectiveisenough}, if \eqref{bona}
        holds for $\Datum$, then $m^G_C$ is injective for general $C$
        and therefore $m^K_F$ is injective for general $F$, so
        $N(\dd) = S^2H^0(K_F)^K \leq H^0(2K_F)^K $.  But since
        $g(F) \geq 2$, the discussion in \ref{uff} shows that
        $ N(\dd \geq s-3=H^0(2K_F)^K $. Thus $N(\dd) = s-3$,
        i.e. $\dd$ satisfies \eqref{bona}.
      \end{proof}

      \section{Avoiding unnecessary computations}
      \label{sec:tricks}

      This section collects several results that allow to rule out
      \emph{a priori} various cases avoiding some parts, sometimes
      really substantial, of the computation. We briefly explain its
      contents.

      Lemmata \ref{kulka} and \ref{bounds} use the same ideas
      underlying the proof of the Hurwitz theorem to ensure that
      signatures exist only in some ranges.  Theorem \ref {mona})
      summarizes results of Moonen and Mohajer-Zuo, saying that no new
      counter-examples exist in certain cases.

      In \ref{ssg} we introduce spherical systems of generators,
      recall the Chevalley-Weil formula, define refined passports and
      show that $N(\Datum)$ only depends on the refined passport of
      the generators.  We then recall Eichler's formula. It is used in
      the proof of Theorem \ref{Z/2}, which says that no
      counter-example exists with $G=(\Zeta/2\Zeta)^k$ for $g\geq 4$. Its
      Corollary \ref{strongerthanall} is the main tool to cut down the
      number of computations to be done.  Other such tools are
      Frobenius' test (Corollary \ref {Frobby}) and an elementary
      observation on the abelianization of a group admitting a spherical system of generators (\S  \ref{abelianizzo}).

      \begin{lemma}
        \label{kulka}
        If $\datum$ is an $r$-datum of genus $g$ and $G$ contains an
        element of order $ > 4(g-1)$, then either $r=3$, i.e. the
        family is $0$-dimensional, or it coincides with family (5) in
        \cite[Table 2]{fgp}.
      \end{lemma}
      If $x \in G$ has order $>4(g-1)$, then by definition
      $H:=\sx x \xs $ is a \emph{large automorphism group} of $C$.  So
      the Lemma follows immediately from Proposition 4.5 in
      \cite{kulkarni}.  The idea of using upper bounds for the order
      of single elements of $G$ comes from Corollary 5.10 in
      \cite{BP16}, where the classical bound of Wiman was used.  The
      theorem of Kulkarni that we use here is more precise.

      \begin{lemma}
        \label{bounds}
        Let $\Datum= \datum$ be an $r$-datum with genus $g\geq 2$ and
        $r \geq 4$.  If the datum corresponds to an action of $G$ on a
        smooth curve $X$ with $X/G = \PP^1$, then (a) $r\leq 2g+2$
        with equality only for $X$ hyperelliptic and $G$ generated by the hyperelliptic involution, (b)
        $r \leq 4 + \frac{4(g-1)}{d}$ and (c) $|G|\leq 12 (g-1)$.
      \end{lemma}
      \begin{proof}
        The arguments are extremely classical, but for the reader's
        convenience we give the proof.  Set $d:=|G|$,
        $\delta: = \sum_{i=1}^r \frac{1}{m_i}$ and
        $\mu: = r -2 - \delta$.  By the Riemann-Hurwitz formula,
        \begin{gather}
          \label{RH2}
          2 (g-1) = d \cd \mu
        \end{gather}
        Assume $2 \leq m_1 \leq m_2 \leq \cds \leq m_r$.  Since
        $g \geq 2$, $ \mu >0$. For $x>0$ set $f(x):= 1 - 1 /x$. Then
        $\mu = \sum_{i-1}^r f(m_i) -2$. Since $f$ is increasing
        $ \mu \geq r \cd f(2) -2 = (r-4)/{2}$.  Using $d\geq 2$ and
        \eqref{RH2}, this gives $ g -1 \geq (r-4)/2$, i.e. the
        inequality in (a).  If equality holds $|G|=2$, so the curves
        are hyperelliptic. By a dimensional count the family coincides
        with that of hyperelliptic curves. This proves (a).

        Set
        $A(r):= \{x=(x_1, \lds, x_r) \in \Zeta^r: x_i \geq 2,
        \sum_{i-1}^r f(x_i) >2\}$ and
        \begin{gather*}
          \bar{\mu}(r):= \min_{x\in A(r)} \biggl\{ \sum_{i=1}^r f(x_i)
          -2 \biggr \}
        \end{gather*}
        Using the fact that $f$ is strictly increasing one verifies
        that for $r=4$ the minimum is achieved at $x=(2,2,2,3) $ and
        $\bar{\mu}(4) = 1/6$, while for $r \geq 5$ the minimum is
        achieved at $x=\underbrace{(2, \dots, 2)}_{r \text{ times}}$
        and $\bar{\mu}(r) = r/2 -2$.  So for any $r\geq 4$ we have
        $\bar{\mu}(r)\geq (r-4)/2$.  Let now $\mm$ be the signature of
        the datum $\datum$. Then $\mm \in A(r)$, so
        $\mu \geq \bar{\mu}(r)$.  Thus \eqref{RH2} gives
        $2 (g-1)/ {d} \geq \bar{\mu}(r) \geq (r-4)/{2}$, which is the
        inequality in (b).

        If $r=4$, \eqref{RH2} gives
        ${2 (g-1)}/ {d} \geq \bar{\mu}(4) = 1/6$, which is equivalent
        to the inequality in (c).  If $r>4$ in the same way we get
        $ d \leq {4(g-1)}/{(r-4)} \leq 4 (g-1)$.
        But $ {4(g-1)}/{(r-4)} \leq 4 (g-1) \leq 12(g-1)$.  Hence the
        inequality in (c) holds for every value of $r$.
      \end{proof}

      \begin{teo}
        \label{mona}
        The data $\Datum=(G,\theta) $ satisfying \eqref{bona} with $G$
        cyclic or with $G$ abelian and $r=4$ are Hurwitz equivalent to
        those mentioned in Theorem \ref{data}.  Moreover for such data
        \eqref{bona} is necessary for $\zd$ to be a Shimura
        subvariety.
      \end{teo}
      These results are due to Moonen \cite{moonen-special} and
      Mohajer-Zuo \cite[Thms. 3.1 and 6.2]{mozuo}.

      \begin{say}\label{ssg}
        If $G$ is a finite group, giving an $r$-datum
        $\Delta=(G,\theta)$ is equivalent to giving a list of
        generators $g_1, \lds, g_r$ of $G$ such that $g_i\neq 1$ for
        any $i$ and subject to the constraint $g_1\cds g_r = 1$.
        Indeed, this defines an epimorphism
        $\theta\colon \Gamma_r\to G$ by $\theta(\ga_i)=g_i$.  From now
        on we will write $\Datum \in \datar$ as
        $\Datum=(G,g_1, \lds, g_r)$, and we will call
        $(g_1, \lds, g_r)$ a \emph{spherical system of generators} of
        the group $G$.

        Let $\chi_\rho$ denote the character of the representation
        $\rho$ defined in \eqref{defrho}.  As explained in \cite[\S\S
        2.9ff]{fgp} the number $N(\Delta)$ in \eqref {defN} can be
        computed from $\chi_\rho$:
        \begin{gather}
          \label{N}
          N(\Delta)=\frac{1}{2|G|} \sum_{a \in G } \bigl (
          \chi_\rho(a^2) + \chi_\rho(a)^2 \bigr).
        \end{gather}
        So to test \eqref{bona} one needs to compute $\chi_\rho$ for a
        datum $\Delta$.  There are two ways to do that: using
        Eichler's trace formula or the Chevalley-Weil formula. We need
        both and we start from the Chevalley-Weil formula.
      \end{say}
    
      \begin{say} \label{CWsay} Next, fix a datum
        $\Delta= (G, g_1, \lds, g_r)$ and let $m_j:=\ord(g_j)$ as
        usual.  Denote by $\irr G$ the set of irreducible characters
        of $G$. For each $\chi \in \irr G $ fix a representation
        $\sigma_\chi$ with character $\chi$.  For $n\in \mathbb{N}$,
        $n>0$ set $\zeta_n := \exp(2\pi i /n)$. If $\chi \in \irr G$,
        $1 \leq j \leq r$ and $0 \leq \alf < m_j$, denote by
        $N_{j,\alf}$ the multiplicity of $\zeta_{m_j}^\alf $ as an
        eigenvalue of $\sigma _\chi ( g_j)$.
      \end{say}

\begin{teo}[Chevalley--Weil]
  If $\Delta=(G, g_1, \lds, g_r) $ is a datum for the Galois covering
  $ C \rightarrow \PP^1$, then the multiplicity $\mu_{\chi}$ of
  $\chi \in\irr G$ in $\rho$ is
  \begin{equation}\label{eq_ChevWeilFormula}
    \mu_{\chi} = -\deg {\chi} +\sum^r_{j=1} \sum^{m_j-1}_{\alpha=0} N_{j,\alpha}\frac{\alpha}{m_j}
    + \eps,
  \end{equation}
  where $\eps=1$ if $\chi$ is the trivial character and $\eps=0$
  otherwise.
\end{teo}

A nice reference for the Chevalley-Weil formula is
\cite[Ch. 1]{gl}. Our implementation uses this formula to compute
$\chi_\rho$ and hence $N(\Delta)$.  In fact we use the same algorithm
as Glei\ss{}ner, which is based in turn on \cite{swinarski}, but with
code optimized for our setting (see \ref{wasgleissners}).

\begin{definition} \label{defrefined} Given a finite group $G$ let
  $\mathcal{C}_G$ or simply $\mathcal{C}$ denote the set of conjugacy
  classes of $G$. The symmetric group $\Sigma_r$ acts on
  $$
  \mathcal{C}_G^r:= \underbrace{\mathcal{C}_G \times \dots \times
    \mathcal{C}_G}_{r \text{ times}} .
  $$ A \emph{refined passport}
  with $r$ branch points for the group $G$ is an element of
  $\mathcal{C}^r_G / \Sigma_r$. Thus a refined passport is an
  undordered sequence of conjugacy classes of $G$.  Given
  $\Datum=(G,g_1, \lds, g_r)$, the \emph{refined passport} of $\Datum$
  is the class of $ ([g_1], \lds, [g_r])$ in
  $\mathcal{C}^r_G/\Sigma_r$.
\end{definition}
Note that this definition is slightly different from those of
\cite{lando-zvonkin} and \cite{paulhus}: we do not assume that a
refined passport comes from a datum.

\begin{say} \label{sayrefined} It is clear that the numbers
  $N_{j,\alf}$ defined in \ref{CWsay} do not change if $g_j$ is
  replaced by another element $g_j'\in G$ which is conjugate to $g_j$.
  Another observation is that obviously the sum in
  \eqref{eq_ChevWeilFormula} is independent of the order.  Thus
  $N(\Delta)$ depends only on the refined passport of $\Datum$.  This
  elementary observation is at the basis of our approach to the
  computation.
\end{say}

\begin{lemma}
  \label{cccp}
  Let $G$ be a finite group and let $C_i \in \mathcal{C}_G$ for
  $i=1, \lds, r$.  Assume that there is a datum
  $\Datum=(G, g_1, \lds, g_r)$ with $g_i \in C_i$ for $i=1, \lds, r$.
  Then for any $\sigma=(\sigma_1, \lds, \sigma_r) \in \Sigma_r$ there
  is a datum $(G, \ga_1, \lds, \ga_r)$ such that
  $\ga_i\in C_{\sigma_i}$ for $i=1, \lds, r$.
\end{lemma}
\begin{proof}
  Since $\Sigma_r$ is generated by simple transpositions, it is enough
  to prove the result for $\sigma = (j, j+1)$, $1\leq j < r$.  Set
  \begin{gather*}
    \ga_i = g_i, \text{ for }i\not \in \{ j, j+1\}, \qquad \ga_j = g_j
    g_{j+1}g_j\meno , \qquad \ga_{j+1} = g_j.
  \end{gather*}
  Then $(G, \ga_1, \lds, \ga_r)$ is still a datum and
  $\ga_i \in C_{\sigma_i}$ for any $i$.
\end{proof}

\begin{say}
  \label{Eichlersay}
  Now we turn to Eichler's formula, which is important to rule out a
  class of groups.  Recall that if $a \in G$, $p\in C$ and
  $a\cd p =p$, then $da(p) $ $ \in$ $ \End T_pC =\C$ is 
  a root of unity, see e.g. \cite[p. 106]{FK}. 
\end{say}
\begin{teo} [Eichler Trace Formula]
  If $a \in G $, $a \neq 1$ then
  \begin{equation}\label{eq_EichlerFormula}
    \chi_\rho(a )=1-\sum_{p\in \fix(a)}\frac{1}{1-da(p)}.
  \end{equation}  
\end{teo}
See e.g.  \cite[Thm. V.2.9, p. 281]{FK}.

\begin{teo}\label{Z/2}
  Let $\Delta =(G, g_1, \lds, g_r)$ be a datum corresponding to a
  covering $C \ra \PP^1$ with $G \cong \left( \Zeta/2\Zeta \right)^k$.
  If $g(C) \geq 4$, then \eqref{bona} does not hold for $\Delta$.
\end{teo}
\begin{proof}
  The families fulfilling condition \eqref{bona} with genus up to $7$
  have been classified in \cite[Theorems 5.4 and 5.5]{fgp} and are all
  listed in \cite[Table 2]{fgp}: inspecting the table we see that we
  may assume $g(C) \geq 8$.

  Since all elements $a$ in $G$, $a \neq 1$, have order $2$, by the
  Hurwitz formula
  \begin{equation*}
    \chi_\rho(1)=g(C)=1+\frac{|G|}4 (r-4)=1+2^{k-2}(r-4).
  \end{equation*}
  Moreover for all $p \in \fix (a)$, $da(p)=-1\in \R$ and then, by
  \eqref{eq_EichlerFormula} for all $a \in G$, $ \chi_\rho(a )\in \R$.
  In particular all summands in the expression of $N$ in \eqref{N} are
  real numbers and
  \begin{multline*}
    N(\Delta)= \frac{1}{2|G|} \sum_{a \in G } \bigl ( \chi_\rho(a^2) +
    \chi_\rho(a)^2 \bigr) = \frac{1}{2^{k+1}} \sum_{a \in G } \bigl (
    \chi_\rho(1) + \chi_\rho(a)^2 \bigr)
    \geq\\
    \geq \frac{\bigl (\sum_{a \in G } \chi_\rho(1)\bigr) +
      \chi_\rho(1)^2 }{2^{k+1}} =
    g(C)\left( \frac12 + \frac{g(C)}{2^{k+1}}  \right)=\\
    = g(C)\left( \frac12 + \frac{1}{2^{k+1}} + \frac{r-4}8\right)=
    g(C) \left( \frac1{2^{k+1}} +\frac{r}8 \right) > g(C) \left(
      \frac{r}8 \right)\geq r
  \end{multline*}
  contradicting \eqref{bona}.
\end{proof}
Considering Lemma \ref{quoziente} we deduce the following stronger
result:
\begin{cor}
  \label{strongerthanall}
  Let $\Delta =(G, g_1, \lds, g_r)$ be a datum corresponding to a
  covering $C \ra \PP^1$.  If there is a surjective map
  $G \rightarrow (\Zeta/2\Zeta)^4$, then \eqref{bona} does not hold
  for $\Delta$.
\end{cor}
\begin{proof}
  Assume by contradiction that \eqref{bona} holds.

  Let $H$ be the kernel of the surjection
  $G \rightarrow (\Zeta/2\Zeta)^4$ and consider the family of the
  curves $F_t=C_t/H \rightarrow \PP^1$ as in \ref{sayquotientcovers}.
  They are Galois covers with datum
  $\dd=((\Zeta/2\Zeta)^4, h_1, \cdots, h_s)$.

  Since each set of generators of $(\Zeta/2\Zeta)^4$ has cardinality
  at least $4$, then $s\geq 5$. This implies $g(F) \geq 2$ by the
  Hurwitz formula and $g(F) \leq 4$ by Lemma \ref{quoziente} and
  Theorem \ref{Z/2}.

  The Galois covers of $\PP^1$ with genus among $2$ and $4$ having $4$
  or more branch points are listed in \cite[Table 2]{fgp}: we see that
  the group $(\Zeta/2\Zeta)^4$ does not occur, reaching an absurd.
\end{proof}

The Galois group $G$ of family (34) in \cite[Table 2]{fgp} admits
$(\Zeta/2\Zeta)^3$ as a quotient.  Thus one cannot improve the above
Corollary by substituting $(\Zeta/2\Zeta)^4$ with one of its proper
quotients.  In fact applying Lemma \ref {quoziente} to this case
yields one of the families of elliptic curves mentioned after Theorem
\ref{data}.

There is another useful criterion, already used by Breuer
\cite{breuer} and Paulhus \cite{paulhus}.  Indeed, for some elements
$c$, one can ascertain \emph{a priori} that
$\pi^{-1}(c)=p^{-1}(\tilde c)$ does not contain any system of
generators at all. This is based on a theorem of Frobenius. (See
\cite[p. 406]{lando-zvonkin} for a proof.)
\begin{teo}[Frobenius' formula]
  \label{Frobenius}
  Given a finite group $G$ and conjugacy classes $C_1,\dotsc, C_r$,
  the number of $r$-ples
  $(g_1,\dotsc, g_r)\in C_1\times \dots \times C_r$ such that
  $\prod g_i=1$ is
  \[\frac{|C_1|\dotsm |C_r|}{|G|} \sum_{\chi \in \irr G}
    \frac{\chi(C_1)\dotsm \chi(C_r)}{\chi(1)^{r-2}}.\]
\end{teo}
Notice that this condition is independent of the order.

\begin{cor}
  \label{Frobby} Let $G$ be a group and $(C_1 \lds, C_r)$ a refined
  passport.  If
  \begin{gather*}
    \sum_{\chi \in \irr G} \frac{\chi(C_1)\dotsm
      \chi(C_r)}{\chi(1)^{r-2}} =0,
  \end{gather*}
  then there is no datum $(G,g_1, \lds, g_r) $ with refined passport
  $(C_1, \lds, C_r)$.
\end{cor}

\begin{say}
  \label{abelianizzo}
  We conclude with a useful elementary observation.  Assume that a
  group $G$ admits a system of spherical generators $(g_1, \lds, g_r)$
  with signature $(m_1, \lds, m_r)$.  Decompose its abelianization
  $\Ab G = \Zeta/k_1\Zeta \oplus \cds \oplus \Zeta/{k_p}\Zeta$ with
  $k_1 | \cds | k_p$ (i.e. the $k_i$'s are the invariant
  factors). Since for any $j$, $\Ab G$ is generated by the images of
  $g_1 , \lds, \hat{g_j}, \lds, g_r $, it follows that $p \leq r-1$
  and that
  $k_p$ divides $\operatorname{lcm}(m_1,\dots, \hat{m}_j,\dots, m_r)$ for
  any $j$.
\end{say}

  \section{The algorithm}
    \label{sec:algoritmo}
\begin{say}
  Given a group $G$, let $\mathcal{C}_G$ be the set of its conjugacy
  classes. Recall from Definition \ref{defrefined}
  that a refined passport on $G$ with $r$ branch points is an
  unordered sequence of $r$ conjugacy classes of $G$, i.e. an element
  of $\mathcal{C}_G^r/\Sigma_r$. If a refined passport contains a
  spherical system of generators $\Delta=(g_1,\dotsc, g_r)$,
  $g(\Delta)$ and $N(\Delta)$ only depend on the refined passport of
  $\Delta$. We will say that a refined passport is a
  \emph{counter-example of genus $g$} if it contains a spherical
  system of generators $\Delta$ with $g(\Delta)=g$ such that
  \eqref{bona} holds. Notice that refined passports that satisfy \eqref{bona}
  formally but do not contain a spherical system of generators are
  excluded by this definition.  The group $\Aut G$ acts both on
  $\mathcal{C}_G$ and on the set of refined passports.

\end{say}

\begin{say}

We illustrate an algorithm to attack the following:
\end{say}
\begin{problem}
\label{problem:classification}
For fixed $g\geq 2$, list groups $G$ and counter-examples of genus $g$ on $G$ with $r\geq 4$ branch points, one for each orbit of $\Aut(G)$, leaving aside those with   $G$ cyclic and those with $G$ abelian and $r=4$.
\end{problem}
Our basic strategy is to fix $r$, and then choose one refined passport of genus $g$ with $r$ branch points in each $\Aut(G)$-orbit. If \eqref{bona} holds, it then suffices to determine whether the refined passport contains a system of  spherical systems of generators.
\begin{say}
 
As in \cite{BCGP,fgp}, we use signature as an invariant. Using the notation of Definition \ref{def:data} signature
defines a map
\begin{equation*}
   \datar \ra \N^r, \quad (g_1, \dots, g_r)\mapsto ( \ord(g_1), \lds, \ord(g_r)).
 \end{equation*}
Since the order of an element only depends on its conjugacy class, the signature of a spherical system of generators $(g_1,\dotsc, g_r)$ only depends on the conjugacy classes $([g_1],\dotsc, [g_r])$. Corresponding to the fact that refined passports are taken up to reordering (Lemma \ref{cccp}), signatures can be considered up to permutation, i.e. we can restrict to signatures satisfying $m_1\leq \dots \leq m_r$.

We iterate over the order $d=|G|$. For fixed $d$, let $\mathfrak{S}_{d,g}$ be the set of finite sequences $\mm=(m_1,\dotsc, m_r)$ such that
\begin{enumerate}[label=(S\arabic*)]
\item \label{item:S1}
 $4\leq r\leq \frac{4(g-1)}d+4$ and $d\leq 12(g-1)$;
\item \label{item:S2} each $m_i$ is a divisor of $d$;
\item \label{item:S3} $1< m_i< d$;
\item \label{item:S4} $g$ and $\mm$ satisfy \eqref{RH};
\item \label{item:S5}$m_1\leq \dots \leq m_r$;
\end{enumerate}

By Lemma~\ref{bounds}, the signature of a spherical system of
generators $\Delta$ with $r\geq 4$ and $g(\Delta)=g$ must satisfy
\ref{item:S1}; the restriction $r\geq 4$ ensures that the family is
positive-dimensional, see \eqref{r3}; the restriction $m_i<d$ in
\ref{item:S3} is motivated by the fact that we are only interested in
noncyclic groups $G$.
\end{say}

The set of ``admissible'' signatures $\mathfrak{S}_{d,g}$ is computed by Algorithm~\ref{algo:signatures}.
In the implementation, we found it convenient to compute each $\mathfrak{S}_{d,g}$ for $2\leq g\leq g_{max}$ simultaneously, and then store the result on disk for later retrieval, rather than iterate over $g$; this prevents repeating some computations.

\begin{say}
Elements of $\mathcal{C}_G^r/\Sigma_r$ (i.e. refined passports) can be viewed as \emph{multisets}. Given a set $X$, a multiset of elements of $X$ can be defined as a set $\{(x_1,n_1),\dotsc, (x_k,n_k)\}$ where the $x_i$ are pairwise disjoint elements of $X$ and the $n_i$ are nonnegative integers representing the \emph{multiplicity} of $x_i$. In fact, it is customary to require the $n_i$ to be positive, but it will be convenient for our purposes to allow them to be zero as well. We will write a multiset as $\{x_1^{n_1},\dotsc, x_k^{n_k}\}$. A set $\{x_1,\dotsc, x_k\}$ can be identified with the multiset $\{x_1^1,\dotsc, x_k^1\}$, and the union of two multisets is defined in the obvious way by adding multiplicities.

It will also be convenient to represent elements of $\mathfrak{S}_{d,g}$ as multisets of integers $\{m_1^{n_1},\dotsc, m_k^{n_k}\}$; for instance, the signature $(2,2,3,3,3)$ will be represented by the multiset $\{2^2,3^3\}$.
\end{say}

\begin{say}
  Problem~\ref{problem:classification} can then be addressed by
  iterating through the signatures $\mm\in\mathfrak{S}_{d,g}$ computed in
  Algorithm~\ref{algo:signatures} and groups $G$ of order $d$. A refined passport with signature $\mm$ only exists on a group $G$ if there is at least one element of order $m_j$ for every $m_j\in\mm$; we therefore discard groups and signatures that do not satisfy this condition. More groups and signatures can be eliminated by taking advantage of
  Lemma~\ref{kulka}, 
  Corollary~\ref{strongerthanall} and the observation in
  \ref{abelianizzo}. This
procedure is displayed in Algorithm~\ref{algo:fixed_g}, which reduces
the problem to identifying counter-examples for fixed group and
signature. Notice that on line~\ref{line:multiset} the signature
$\{m_1^{n_1},\dotsc, m_k^{n_k}\}$ is converted into a multiset
\begin{equation}
 \label{eqn:multisetofsets}
 \{A_1^{n_1},\dotsc, A_k^{n_k}\}\subset\mathcal{P}(\mathcal{C}_G),
\end{equation}
where each $A_i$ is the subset of $\mathcal{C}_G$ of conjugacy classes of order $m_i$. This is the basis for the recursion of Algorithm~\ref{alg:fixedGm}.
\end{say}

\begin{say}
  \label{wasgleissners}
At this point we need to determine the counter-examples with a given signature $\mm$ and group $G$. This is achieved by picking one refined passport with signature $\mm$ in each $\Aut(G)$-orbit, then verifying whether \eqref{bona} holds and the refined passport contains a spherical system of generators.

The iteration through one refined passport in each $\Aut(G)$-orbit is performed in Algorithm~\ref{alg:fixedGm}. A refined passport with signature $\{m_1^{n_1},\dotsc, m_k^{n_k}\}$ is obtained by choosing $n_i$ conjugacy classes with order $m_i$ for each $1\leq i\leq k$; in terms of \eqref{eqn:multisetofsets}, for each $i$ we must choose a multiset $S_i$ of $n_i$ elements of $A_i$, counted with multiplicities. We can write $S_i$ in a unique way as a union of sets $\bigcup_j B_{ij}$, where  
$B_{i1}\supset  B_{i2}\supset\dots$ is a definitely empty sequence of subsets of $A_i$; this means that the multiplicity of $C$ in $S_i$ is the number of indices $j$ such that $C$ is in $B_{ij}$. Thus, iterating through the possible multisets $S_i$ is equivalent to iterating through sequences
\[A_i\supset B_{i1}\supset  B_{i2}\supset\dots,\quad  \sum \abs{B_{ij}}=n_i.\]
This must be repeated for each $i=1,\dotsc, k$.

Our goal is to perform a similar iteration by choosing a single element in each $\Aut(G)$-orbit. To begin with, our algorithm picks a subset $B$ of $A_k$ with $1\leq h\leq n_k$ elements, representing $B_{k1}$ in the notation above. For each choice of $B$, the function recursively iterates through refined passports obtained by taking the union of $B$ and a refined passport with $n_i$ elements in each $A_i$, $i<k$ and $n_k-h$ elements in $B$. The recursive call iterates through one refined passport for each $H$-orbit, where $H$ is the stabilizer of $B$ in $\Aut(G)$. Top-level iteration over one subset $B$ for each $\Aut(G)$-orbit completes the algorithm.

This approach requires a much  lower amount of memory than determining all possible refined passports first and then picking one in each $\Aut(G)$-orbit. Notice also that the refined passports produced by the algorithm are elaborated sequentially, and not stored simultaneously into memory. Nevertheless, the algorithm must iterate through one subset of $A_k$ for each $\Aut(G)$-orbit, and we are not aware of any efficient way of doing this without storing all subsets of fixed cardinality in memory. This is the one point in the whole algorithm where memory consumption can be significant.

Algorithm~\ref{algo:counterornothing} determines whether a refined passport is a counter-example; first, the condition of  Theorem~\ref{Frobenius} is verified, i.e. whether $\sum_{\chi} \frac{\chi(C_1)\dotsm \chi(C_r)}{\chi(1)^{r-2}}$ is nonzero; if so,  we will say that $(C_1,\dotsc, C_r)$ passes Frobenius' test. 
Then, condition \eqref{bona} is tested by selecting random elements inside each $C_i$ and computing $N(\Delta)$ by \eqref{eq_ChevWeilFormula}. Notice that  each term $\sum_\alpha N_{j,\alpha}\alpha/m_j$ appearing in \eqref{eq_ChevWeilFormula} only depends on the corresponding $g_j$, and the characters $\chi$ only depend on the group $G$. Thus, it suffices to compute these data at the beginning of the computation, when $G$ is fixed, making the computation of \eqref{eq_ChevWeilFormula} in the iteration quite fast. Only when both Frobenius' test and \eqref{bona} hold does the algorithm
perform the most computationally expensive step, namely checking whether $C_1\times\dots\times C_r$ contain a spherical system of generators, by straightforward iteration.
\end{say}

\begin{say}
For abelian groups $G$, conjugacy classes contain a single element, and the algorithm can be improved.

First, observe that Frobenius' test is useless in this case: the product $C_1\times \dots \times C_r$ contains a single element $(g_1,\dotsc, g_r)$, so the condition $\prod g_i=1$ is best verified directly.

Second, since refined passports contain a single element of $G^r$, we effectively iterate through elements of $G^r$. However, in a spherical system of generators $(g_1,\dotsc, g_r)$ any element is determined by the others, so we can iterate through ``short sequences'' $(g_1,\dotsc,\hat{g_j},\dotsc,  g_{r})$. Thus, we proceed as follows.

We fix $m_j$ in $\mm=(m_1,\dotsc, m_r)$  such that the number of elements of $G$ with order $m_j$ is largest; then, we use a scheme analogous to Algorithm~\ref{alg:fixedGm} to iterate through $(r-1)$-ples $(g_1,\dotsc,\hat{g_j},\dotsc,  g_{r})\in G^{r-1}$ with signature $(m_1,\dotsc, \hat{m_j},\dotsc, m_r)$, one for each $\Aut(G)$-orbit. We then define $g_j$ as the inverse of $g_1\dotsm \hat{g_j}\dotsm g_r$; if $g_j$ has order $m_j$, the $(g_1,\dotsc, g_r)$ is a candidate for a spherical system of generators with signature $\mm$. At this point, we test condition \eqref{bona} and, if it holds, whether the elements $g_1,\dotsc, g_r$ generate the group $G$.

\begin{algorithm}[tbh]  
\SetKwInOut{Input}{input}
\SetKwInOut{Output}{output}
\SetKwProg{Fn}{Function}{}{}
\Input{integers $g\geq 2,d\geq 2$}
\Output{the set of signatures $\mathfrak{S}_{d,g}$}
\Fn{$\mathfrak{S}_{d,g}(d,g)$}{
\If{$d$ prime}{\Return $\emptyset$ \tcp{\ref{item:S3} cannot be satisfied}}
$\mathfrak{S}_{d,g}\leftarrow\emptyset$\\
  \For{$r$ satisfying \ref{item:S1}}{
  $D\leftarrow\{n\in\N\mid 2\leq n<  d,\; n \text{ divides } d\}$\;
  \For{$m_1,\dotsc, m_r\in D$, $m_1\leq \dots \leq m_r$}{
  \If{$(m_1, \dotsc, m_r)$ satisfies \ref{item:S4} and \ref{item:S5}}
  {insert $(m_1,\dotsc, m_r)$ in $\mathfrak{S}_{d,g}$}
  }}
  \Return $\mathfrak{S}_{d,g}$
}
\caption{\label{algo:signatures}Computing the signatures}
\end{algorithm}

\begin{algorithm}[tbh]
\SetKwInOut{Input}{input}
\SetKwInOut{Output}{output}
\SetKwProg{Fn}{Function}{}{}
\SetKw{Return}{return}
\Input{an integer $g\geq 2$}
\Output{counter-examples of genus $g$ with $r\geq 4$ branch points, one for each $\Aut(G)$-orbit}
\Fn{admissible($G$,$\mm$)}{
$\mathcal{O}\leftarrow\{\ord(g)\mid g\in G\}$\;
\uIf{$r=4$ and $G$ abelian}{\Return false;}
\uElseIf{$r >4$ and $G$ cyclic}{\Return false;}
\uElseIf{some $m_i$ is not in $\mathcal{O}$}{\Return false}
\uElseIf{$g>2$ and some $o\in\mathcal{O}$ is greater than $4(g-1)$}{\Return false \tcp{Lemma~\ref{kulka}}}
\Else{
         decompose the abelianization of $G$ as $\Zeta/{k_1}\Zeta\oplus \dots \oplus \Zeta/{k_p}\Zeta$, with each $k_i$ dividing $k_{i+1}$\;
         \uIf{at least $4$ elements in $(k_1,\dots, k_p)$ are even}{\Return false \tcp{$G$ surjects over $(\Zeta/2\Zeta)^4$ (Corollary~\ref{strongerthanall})}}
         \uIf{$p\geq r$}{\Return false \tcp{$r-1$ elements cannot generate $G$}}        
         \uElseIf{exists $j$ such that $k_p\nmid\operatorname{lcm}(m_1,\dots, \hat{m}_j,\dots, m_r)$}{\Return false \tcp{\S \ref{abelianizzo}}} 
         \Return true \tcp{passed all tests}
}}

\For{$2\leq d\leq
12(g-1)$} {
    determine $\mathfrak{S}_{d,g}$ by Algorithm~\ref{algo:signatures}\;
        \For{$\mathbf{m}=\{m_1^{k_1},\dotsc, m_k^{r_k}\}$ in $\mathfrak{S}_{d,g}$} {    \For{$G$ group of order $d$} {
\If{admissible($G$,$\mm$)}{
    \For{$1\leq i\leq k$}{\label{line:multiset}
    $A_i\leftarrow \{C\in \mathcal{C}_G\mid \ord(C)=m_i\}$}
    CounterExamplesIn $(G, \{ A_1^{r_1},\dotsc, A_k^{r_k}\})$
    \tcp{Find counter-examples for group $G$ and signature $\mm$ using Algorithm~\ref{alg:fixedGm}}
}}
}
}
 \caption{\label{algo:fixed_g} Find counter-examples of genus $g$ with $r\geq 4$ branch points}
\end{algorithm}

\begin{algorithm}[H]    
\SetKwInOut{Input}{input}
\SetKwInOut{Output}{output}
\SetKw{Return}{return}
\SetKwComment{tcp}{//}{}
\SetKwProg{Fn}{Function}{}{}
\Input{a group $G$, a refined passport $(C_1,\dotsc, C_r)$}
\Output{true if the refined passport is a counter-example, false otherwise
}
\caption{Determine whether a refined passport is a counter-example \label{algo:counterornothing}}
\Fn{IsCounterExample($G$, $C_1,\dotsc, C_r$)}{
 \If{$(C_1,\dotsc, C_r)$ passes Frobenius' test and $N= r-3$}{
\For{$(g_1,\dotsc, g_{r-1})$ in $C_1\times\dotsm \times  C_{r-1}$}{
    $g_r\leftarrow (g_1\dotsm g_{r-1})^{-1}$\;
    \If{$g_r\in C_r$ and $\langle g_1,\dotsc, g_{r-1}\rangle =G$}{
        \Return true
}}
 \Return false
} 
}
\end{algorithm}

\begin{algorithm}[H]
\caption{Find counter-examples for fixed group and signature \label{alg:fixedGm}}
\SetKwInOut{Input}{input}
\SetKwInOut{Output}{output}
\SetKwProg{Fn}{Function}{}{}
\SetKw{Return}{return}
\Input{A group $G$ and a nonempty multiset $\{A_1^{n_1},\dots, A_k^{n_k}\}$ where each $A_i$ is a nonempty set of conjugacy classes of $G$ and each $n_i$ is a nonnegative integer}
\Output{Counter-examples obtained by choosing $n_i$ elements in each $A_i$, one for each $\Aut(G)$-orbit}
\Fn{CounterExamplesIn($G$, $\{A_1^{n_1},\dots, A_k^{n_k}\}$, $S=\emptyset$, $H=\Aut(G)$)}{
\tcp{This is a recursive function using two arguments with default values: $S$ is a multiset of conjugacy classes of $G$; $H$ is a subgroup of $\Aut(G)$ acting on each $A_i$}
\uIf{$k=0$}{
    \If{IsCounterExample($G$,$S$)}{print $G,S$}        
}
\uElseIf{$n_k=0$}{CounterExamplesIn($G$,$\{A_1^{n_1},\dots, A_{k-1}^{n_{k-1}}\}$,$S$,$H$)}
    \uElseIf{$A_k$ contains a single element $a$}{
        CounterExamplesIn($G$, $\{A_1^{n_1},\dots, A_{k-1}^{n_{k-1}}\}$, $S\cup \{a^{n_k}\}$, $H$)
        }    
\Else{CounterExamplesWith($G$,$\{A_1^{n_1},\dots,A_{k-1}^{n_{k-1}}\}$,$S$,$H$,$A_k$, $n_k$)}
}
\Fn{CounterExamplesWith($G$,$\{A_1^{n_1},\dots, A_k^{n_k}\}$,$S$,$H$,$A$,$n$)}{
\tcp
{Helper function that iterates through subsets of $A$}
    \For{$1\leq h\leq n$}{
        $X\leftarrow \{B\subset A\mid \abs{B}=h\}$\\
        \For{one $B$ in each $H$-orbit of $X$ }{
            $K\leftarrow$ stabilizer of $B$ for action of $H$ on $X$\\
            CounterExamplesIn($G$, $\{A_1^{n_1},\dots, A_{k}^{n_{k}},B^{n-h}\}$,$S\cup B$,$K$)}
}
}

\end{algorithm}

This is clearly faster than a plain application of Algorithm~\ref{alg:fixedGm}, because an $r$-fold iteration is replaced by an $(r-1)$-fold iteration. Notice, however, that the same counter-example can appear more than once in the output, if this method is applied to cases where $m_j$ has multiplicity greater than one, say $m_j=m_{j+1}$. Indeed, a counter-example $(g_1,\dotsc, g_r)$ can be obtained by completing two different short sequences, namely
$(g_1,\dotsc, \widehat{g_j}, \dotsc, g_r)$ and
$(g_1,\dotsc, \widehat{g_{j+1}}, \dotsc, g_r)$. If the two short sequences
lie in different $\Aut(G)$-orbits, the output will contain two counter-examples in the $\Aut(G)$-orbit of $(g_1,\dotsc, g_r)$.
\end{say}

\end{document}